\documentclass{article}

\usepackage{amsmath}
\usepackage{amssymb}
\usepackage{amsthm}
\usepackage{graphicx}
\usepackage{amscd}
\usepackage{epic, eepic}
\usepackage{url}
\usepackage{color}
\usepackage[utf8]{inputenc} 
\usepackage{makebox}
\usepackage{comment}
\usepackage{enumerate}   
\usepackage[colorinlistoftodos,prependcaption,textsize=small]{todonotes}

\theoremstyle{plain}
\newtheorem{theorem}{\bf Theorem}[section]
\newtheorem{lemma}[theorem]{\bf Lemma}
\newtheorem{claim}[theorem]{\bf Claim}
\newtheorem{cor}[theorem]{\bf Corollary}
\newtheorem{proposition}[theorem]{\bf Proposition}

\newtheorem{prop}[theorem]{\bf Proposition}
\newtheorem{conj}[theorem]{\bf Conjecture}
\newtheorem{construction}[theorem]{\bf Construction}
\newtheorem{nota}[theorem]{\bf Notation}

\newtheorem{remark}[theorem]{\bf Remark}
\newtheorem{defi}[theorem]{\bf Definition}

\def\hom{\hbox{\rm hom}}
\def\ex{\hbox{\rm ex}}
\def\sex{\hbox{\rm satex}}

\newcommand\ddp{{\rm dp}}

\newcommand\cG{{\mathcal G}}
\newcommand\cH{{\mathcal H}}

\newcommand\cN{{\mathcal N}}

\title{
Unified approach to the generalized Turán problem and supersaturation}
\author{Dániel Gerbner\footnote{Alfr\'ed R\'enyi Institute of Mathematics. Supported by the Hungarian National Research, Development and Innovation Office -- NKFIH under the grants FK 132060, KKP-133819, SNN 129364 and KH 130371}, Zoltán Lóránt Nagy\footnote{MTA--ELTE Geometric and Algebraic Combinatorics Research Group, 1117 Budapest, P\'azm\'any P.\ stny.\ 1/C, Hungary. Supported by the Hungarian Research Grants (NKFI) No. K 120154 and 134953 and by the J\'anos Bolyai Scholarship of the Hungarian Academy of Sciences.}, Máté Vizer\footnote{Alfr\'ed R\'enyi Institute of Mathematics. Supported by the Hungarian National Research, Development and Innovation Office -- NKFIH under the grant SNN 129364 and KH 130371, by the J\'anos Bolyai Research Fellowship of the Hungarian Academy of Sciences and by the New National Excellence Program under the grant number \'UNKP-20-5-BME-45.} }

\date{}

\begin{document}

\maketitle

\begin{abstract} In this paper we introduce a unifying approach to the generalized Turán problem and supersaturation results in graph theory.
The supersaturation-extremal function $\sex(n, ~F ~: ~m,~ G)$ is the least number of copies of a subgraph $G$ an $n$-vertex graph can have, which contains at least $m$ copies of $F$ as a subgraph.
We present a survey, discuss previously known results and obtain several new ones focusing mainly on  proof methods, extremal structure and phase transition phenomena. Finally we point out some relation with extremal questions concerning hypergraphs, particularly Berge-type results.
\end{abstract}

\section{Introduction}

Tur\'an-type  problems became a central topic with a variety of tools and beautiful results since the early paper of Tur\'an \cite{Turan} on the forbidden subgraph problem concerning cliques. In general, in a forbidden subgraph problem we look for the maximum value $\ex(n,G)$ for the number of $K_2$ subgraphs, i.e. edges, in  $n$-vertex simple graphs having so subgraphs isomorphic to $G$.
  If the maximum value can be determined (asymptotiocally), we are also interested in the structure of the extremal graphs.
The fundamental theorem of Erd\H os, Simonovits and Stone \cite{ES66} points out that the threshold value will be quadratic in $n$ except when $G$ is a bipartite graph, a case where not even the exact order of magnitude is known in general. For more details on the history of the Tur\'an-type problems we refer to the survey of Füredi and Simonovits \cite{survey}. 

From many notable extensions of the Tur\'an function $\ex(n,G)$, in this paper we focus on two generalizations, which have an emerging significance in recent years.  The first one is usually called "Generalized Tur\'an problems".
Let $F$ and $G$ be  arbitrary graphs.  Alon and Shikhelman \cite{Alon} initiated the study of the function $ex(n, F, G)$ which  denotes the
maximum possible number of copies of $F$ in a $G$-free graph on $n$ vertices. 
Note that the case $F = K_2$ being a single edge gives back  the well studied Turán function, and for specific graph pairs $(F, G)$ a number of results were known previously.

Another way to generalize the Turán problem
is to study the so called supersaturation problem (or Rademacher-Turán-type problem).  This was systematically investigated first by Erdős and Simonovits \cite{ES83} in the same time as the pioneer result due to Rademacher (unpublished) revealed  a particular case, see \cite{Erd}.
Here the aim  in general is to determine the minimum number of subgraphs $G$  in $n$-vertex graphs having $m$ edges, in terms of $m$. This function was called the supersaturation function of $G$, and it takes a positive value exactly if $m> \ex(n,G)$.

The first main step towards the general case was studied by Moon and Moser \cite{MM62} when $F$ is a triangle, which  was finally settled by Razborov \cite{Razborov}. A recent notable sharp result in this area is of Reiher \cite{Reiher}, concerning the clique case $G=K_t$ for arbitrary $t\in \mathbb{Z}^+$, which answered the question of Lovász and Simonovits \cite{LS75}. In general, similarly to the basic Turán-function, there is an essential difference in terms of the chromatic number of $G$, that is, we may expect asymptotically sharp results for $\chi(G)>2$ while in the case $\chi(G)=2$ even the exponent of the function can be unclear in certain domains of $m$.

A sharp result in the latter case is only known for graphs $K_{2,t}$ due to the second author \cite{ZNagy} and He, Ma and Yang \cite{He}, based on earlier work of Erdős and Simonovits \cite{ES83}, and the construction of Füredi \cite{fredi}. In general we refer to the surveys of Simonovits  and Füredi--Simonovits \cite{survey, SMiki} and to \cite{Pikh17} concerning the theory of supersaturation.

Both kinds of generalisations are not only interesting on their own but provide essential tools for the original Turán-type problems (see \cite{gjn}) or to apply the container method introduced in \cite{bms,st}.

Our aim is to provide a unifying approach to these two generalisations. 
To do this, we introduce the following function.

\begin{defi}[Supersaturation-extremal function]
The function $\sex(n, F: m, G)$ denotes the minimum number of subgraphs $G$ in an $n$-vertex graph having at least $m$ copies of $F$ as a subgraph.
\end{defi}

Note that in case $F=K_2$, this is the original supersaturation function, while the generalized Turán-problem is to determine the threshold where the function steps up to the positive range from zero. 

Observe that this problem is equivalent to determining the largest number of copies of $F$ if we are given an upper bound $m'$ on the number of copies of $G$. In the case $G=K_2$, we can assume that we have exactly $m'$ edges. This question was first studied by Ahlswede and Katona \cite{AhlKatona} for $F=K_{1,2}$. For further results see e.g. \cite{Alon2,Alon3,ReiherWagner}. In Section \ref{genbo} we show an example how the methods of this area can be generalized to our setting. 

Bollob\'as \cite{bollobas} considered $\sex(n,K_k: m, K_r)$ for $r>k$. Let $t(q,n,k)$ denote the number of copies of $K_k$ in the $q$-partite Tur\'an graph on $n$ vertices. He proved that $\sex(n,K_k: t(q,n,k), K_r)=t(q,n,r)$. Moreover, if we extend these known values of the function $f(m)=\sex(n,K_k: m, K_r)$ in a convex way, we get a lower bound on $f(m)$ for every $m$. If, on the other hand, $r<k$, then the celebrated Kruskal--Katona shadow theorem \cite{katona, kruskal} gives a bound that is sharp for infinitely many $n$ (see e.g. \cite{cowen,frohm} for some improvement for other values of $n$).

Bollob\'as and Nikiforov \cite{bolnik} showed that if $\sum_{v\in V(G)} d^p(v)>(1-1/r)^p n^{p+1}+x$, then $G$ contains more than $$\frac{xn^{r-p}}{p2^{6r(r+1)+1}r^r}$$ copies of $K_{r+1}$. As $\sum_{v\in V(G)} d^p(v)$ is asymptotically the number of copies of the star on $p+1$ vertices in $G$, this can also be interpreted as a supersaturation result for generalized Tur\'an problems. Cutler, Nir and Radcliffe \cite{cnr} recently studied the cases when both $F$ and $G$ are stars and when $G$ is a star and $F$ is a clique. 
Halfpap and Palmer \cite{hp} showed for arbitrary graphs $F$ and $G$ with $\chi(F)<\chi(G)$ that for every $c$ there exists $c'$ such  that if $m>\ex(n,F,G)+cn^{|V(F)|}$, then $\sex(n, F: m, G)>c'n^{|V(G)|}$. This result can be considered as the non-degenerate case of the problem. \\

We are mainly interested in,  and turn our attention  to the so called degenerate case, when both $F$ and $G$ are bipartite graphs. This case has been attracted significant interest since the theory of supersaturation emerged, partly due to the beautiful conjecture of Erdős and Simonovits, and independently, Sidorenko, see  \cite{Conlon, Erd, ES83, cube, Fox, Sido}. In the forthcoming sections we state general bounds, investigate several cases concerning the most common bipartite graphs and summarise the most central results and  methods.

The paper is organized as follows. In Section 2 we introduce our main notations and briefly summarise the main phenomena and methods in the area concerning the order of magnitude of the  supersaturation-extremal function and the structure of the extremal graphs. Section 3 is  devoted to the determination of  $\sex(n, ~H: ~m,~ F)$  in Theorem \ref{spanning} when $H$ has a spanning subgraph consisting of disjoint copies of $F$. This extends a result of Alon concerning the case $F=K_2$. 

From Section 4 to 7, we investigate several ways to generalize the Ahlswede-Katona result by estimating the supersaturation-extremal function for complete bipartite graph, paths,  trees in general or cycles, where the number of certain complete bipartite graphs or paths is given.
$\sex({n, K_{1,s}: m, K_{a,b}}) $ is bounded in Theorem \ref{csillag1}, which happens to be asymptotically sharp for certain ranges of $m$.  $\sex{(n, P_k : m, P_t)} $ is determined asymptotically in Theorem \ref{pathpath} for large enough $m$, when $k-1 \mid t-1$, and the general case of $\sex{(n, P_k : m, P_t)} $ is also discussed.  Then   $\sex{(n, P_{k+1} : m, C_{2k})} $ is bounded from below in Theorem \ref{pathcycle} using the Theorem of Frankl, Kohakayawa, Rödl on the number of disjoint set pairs in uniform set systems and also prove asymptotically sharp bounds on   $\sex{(n, P_{2k+1} : m, K_{2,t})} $ for $m$ large enough in Theorem \ref{pk2t}.

Finally in Section 8 we show an example how these settings provide an easy application to other extremal graph theory results and in Section 9 we point out the connection to extremal hypergraph problems.


\section{General phenomena} 


We set the main notations that will be used throughout the paper.

\begin{nota}{\mbox \ }

\begin{itemize} 

\item  $\cN(F,G)$ 
denotes the number of $F$-subgraphs in the  graph $G$.
 \item For a set $X\subseteq V(G)$ of vertices, $d(X):=d_G(X)$ denotes the co-degree of $X$, i.e. the number of  common neighbours of the vertices of $X$ in $G$. For the degree of a single vertex $y$ we use the standard notion $d(y)$.
 \item $\overline{G}$ denotes the complement of the graph $G$.
 \item $t\cdot F$ denotes the disjoint union of $t$ distinct copies of the graph $F$.
 \item $f(n)\gg g(n)$ means that $\frac{f(n)}{g(n)}\rightarrow \infty$ while $n\rightarrow \infty$.
\end{itemize}
\end{nota}

We also introduce families of graphs that play crucial role in the following.

\begin{nota}
A graph $F$ on $n$ vertices is a {\em quasi-clique} if it is obtained by a clique of size $t$, $n-t-1$ isolated vertices and a vertex which is joint to some subset of the vertices of the clique. We denote by $K_t^*(n)$ the specific $n$-vertex quasi-clique that contains a clique of size $t$ and $n-t$ isolated vertices, and we omit $n$ and use the notation $K_t^*$ if $n$ is clear from the context. A graph is a  {\em quasi-star}, also called co-clique, if it  is the complement of a {\em quasi-clique}.
\end{nota}

We remark that in some cases we will be interested in the smallest $n$-vertex quasi-clique (or quasi-star) $K$ containing at least $m$ copies of some graph $F$. We will usually consider the smallest $K_t^*$ (resp. $\overline{K_t^*}$) containing at least $m$ copies of $F$. We can do this, as we can obtain $K_t^*$ (resp. $\overline{K_t^*}$) by deleting a small number of edges from $K$, and that does not change the asymptotics of the number of copies of $G$.

\smallskip

 Apart from surveying known results as well as proving several new ones, our main focus is on the following aspects: firstly on the methods which are applicable, secondly on the new phenomena concerning the extremal structure. Key features in brief are phase transitions between families of extremal structures appearing in Turán-type problems: quasi-cliques and quasi-stars in problems concerning certain complete bipartite graphs, polarity graphs concerning paths 
 and new structures as well.

The full understanding of the function $\sex(n, F: m, G)$ for a certain pair of graphs $F$ and $G$ consists of the determination  of the threshold $m=m_0$ from where the function takes positive values - this is the generalised Turán problem -  its rate of growth in terms of $m$ from $m_0$  (sparse case) and the behaviour of the function in the domain where a random injection of $|V(F)|$ vertices determines a subgraph isomorphic to $F$ with a positive probability (dense case). An example for a problem concerning this latter case is the Erdős-Simonovits-Sidorenko conjecture.

Mainly, we only investigate the asymptotic behaviour of the function, although we also mention some partial results for the sparse cases in special cases. The reason for this is the following: even in the simplest cases, say, $\sex(n, K_{1,3}:m, K_{1,1})$ some evidences were pointed out \cite{ReiherWagner} which show that the asymptotic results cannot be strengthened in some sense, that is, the asymptotically best configurations may fail to provide the optimal construction in sporadic cases.

A recent paper of  Day and Sarkar \cite{Day} points out even more by disproving a conjecture of Nagy \cite{DNagy}. The conjecture was that for all $G$,
the maximum of  $\cN(G, H)$ over $n$ vertex graphs of given edge density is asymptotically attained on either a quasi-star or a quasi-clique  graph. This result suggests that even to understand the asymptotic behaviour can be a challenging problem. 

Several approaches have been introduced and developed during the investigation of particular cases of  $\sex(n, F: m, G)$. Without  the need to strive for completeness we mention graph limit techniques, algebraic graph theoretic, in particular  spectral results, flag algebra calculus as  leading examples.

We also investigate the possible applications of tiling-type results due to Alon \cite{Alon2}, Jensen-type convexity reasonings in the spirit of Erdős and Simonovits \cite{ES82},  concatenations, matrix inequalities and extremal hypergraphs results.

Speaking of the supersaturation-extremal function, a closely related problem is to prove inequalities between the densities of various bipartite subgraphs in  graphs, when the density is the homomorphism density. The Sidorenko conjecture itself, namely the form  $t(F,G)\geq t(K_2,G)^{|E(F)|}$ states also a supersaturation result since it can be rewritten as follows. Let $F$ be a bipartite graph  and let $G$ be any graph with $n$ nodes and $m=p\binom{n}{2}$ edges, then  the number $\cN(F,G)$ of copies of $F$ in $G$ is at least $p^{|E(F)|}\binom{n}{|V(F)|}+o(p^{|E(F)|}n^{|V(F)|})$.  This statement, which was proved for several graph families \cite{Conlon, Fox, Sido, Szegedy} would show that the supersaturation problem and the problem of homomorphism densities has the same numerical answer asymptotically if $p=\Omega(1)$. The emerging theory of graph limits provided further instances for inequalities between subgraph densities, see \cite{Lov_local, Reverse, Sah}, which can be interpreted similarly in terms of the supersaturation-extremal function, but typically only provide asymptotically good bounds when the subgraph density is high. 



\section{A general bound}\label{genbo}

Alon \cite{Alon2} showed that if $H$ has a perfect matching, then the most number of copies of $H$ in graphs having $k$ edges is asymptotically in the quasi-clique. In our language, it means he  asymptotically determined $\sex(n, H: m, K_2)$, while the case of other graphs $H$ is left open. For example, it implies the result for paths with even number of vertices. For paths with odd number of vertices, the exact result is known for $P_3$ \cite{AhlKatona} and the asymptotic result is known for $P_5$ \cite{DNagy}, but not for longer paths.

Here we show that his simple proof can be generalized to our setting and $K_2$ can be replaced by an arbitrary connected graph $F$ in his result as follows.

 


\begin{theorem}[Supersaturation for tilings]\label{spanning}
Suppose that $t\cdot F$ is a spanning subgraph of $H$ for a suitable $t$. Let $k\gg n^{|V(F)|-1}$. 
Let $r\le n$ be the largest integer such that $K_r$ contains at most $k$ copies of $F$. Assume $G$ is a graph on $n$ vertices and contains at most $k$ copies of $F$. Then $\cN(H,G)\le (1+o(1))\cN(H,K_r)$.
 
Equivalently, let $m\gg n^{|V(H)|-t}$ and $q\le n$ be the smallest integer such that $K_q$ contains at least $m$ copies of $H$. Then $\sex(n, ~H: ~m,~ F)= (1+o(1))\cN(F,K_q^*)$.
 
\end{theorem}

\begin{proof}
We follow the ideas of Alon \cite{Alon2}. Let us choose $H$ by choosing first $t\cdot F$, and then the other edges needed to complete $t\cdot F$ to $H$. Observe first that there are at most $k$ ways to choose a single copy of $F$ in $G$, thus at most $(1+o(1))k^t/t!$ ways to choose $t\cdot F$. 

We claim that there are $(1+o(1))k^t/t!$ ways to choose $t\cdot F$ in $K_r$. There are $k':=\cN(F,K_r)=(1+o(1))k$ ways to choose the first copy of $F$.
Then we pick another copy, disjoint from it. There are at least $k'-|V(F)|n^{|V(F)|-1}$ ways to do this. Indeed, a copy that is not disjoint from the first copy has to contain one of the $|V(F)|$ vertices from the first copy, and there are at most $n$ ways to choose its other vertices. The same way, we can always pick a copy of $F$ disjoint from the previous copies in at least $k'-t|V(F)|n^{|V(F)|-1}=(1+o(1))k$ ways. Thus altogether there are $(1+o(1))k^t/t!$ ways to pick $t\cdot F$ in $K_r$.

It is left to pick the remaining edges. Observe that the number of ways to pick them is maximised by the quasi-clique, as every possible edge we might want to pick is contained in the quasi-clique. Note that as $t\cdot F$ was a spanning subgraph of $H$, we are done with the statement.
\end{proof}

\section{Complete bipartite graphs}

The generalized Tur\'an problem $\ex(n,K_{s,t},K_{a,b})$ was studied for different values of the parameters in \cite{Alon,Szabo,gmv,myz}. Here we will focus on the case $t=1$.

We will need a variant of the power mean inequalities to compare certain functions of binomial coefficients. 

\begin{lemma}\label{technical} Let $d_1, \ldots, d_n$ be non-negative integers and $a\geq s$ be positive integers. Then
$$\frac{1}{n}\sum_{i=1}^n a!\binom{d_i}{a} \geq \left[\left(\sqrt[s]{\frac{\sum_{i=1}^n s!\binom{d_i}{s} }{n}}-a +1\right)_+\right]^a,$$
provided that the average of the numbers $d_i$ is at least $a$. Here $f(x)_+$ means the positive part of the expression, that is, $f(x)_+=\max \{0, f(x)\}$.
\end{lemma}

\begin{proof}
We start with a well known inequality. Suppose that  $x_i\geq q>0$ are real numbers and $s$ is a positive integer then 

\begin{equation}\label{eq:1}
    \frac{1}{n}\sum_{i=1}^n (x_i-q)^s \geq \left( \sqrt[s]{\frac{\sum_i x_i^s}{n}} -q \right)^s.
\end{equation}
 This follows from the triangle inequality applied to the $s$-norm of the vectors  $(x_1-q,\ldots,x_n-q)$ and $(q,\ldots,q)$ of $\mathbb{R}^n$.

Let us assume that all $d_i'$s are at least $a-1$. Now, we have 
$$\frac{1}{n}\sum_{i=1}^n a!\binom{d_i}{a} \geq \frac{1}{n}\sum_{i=1}^n {(d_i-a+1)}^a.$$ 
Using the power mean inequality for the $a$th and $s$th powers of the numbers $(d_i-a+1)$, we obtain 

$$\sqrt[a]{\frac{1}{n}\sum_{i=1}^n {(d_i-a+1)}^a } \geq \sqrt[s]{\frac{1}{n}\sum_{i=1}^n {(d_i-a+1)}^s }.$$
Here we may apply inequality (\ref{eq:1}) with $x_i=d_i$, $q=a-1$, which implies
$$\frac{1}{n}\sum_{i=1}^n a!\binom{d_i}{a} \geq \left( \sqrt[s]{\frac{\sum_i d_i^s}{n}} -a+1 \right)^a.$$
Finally, observe that $d_i^s\geq s!\binom{d_i}{s}$, to conclude the desired inequality.

If the assumption $d_i\geq a-1$ fails on some $d_i$, then take the value $a-1$ instead of  all these $d_i$s. The left hand side would not change, the right hand side could only increase, but for these newly defined values, the lemma holds, so the extension works as well. 
\end{proof}

We would like to understand the behaviour of $\sex(n, K_{1,s}: m, K_{a,b})$.
Note that the case $s=1$ is already covered by the original supersaturation results. One would expect that a Kővári--Sós--Turán type, or in other words, a Jensen-type bound holds, moreover the extremal graph is "balanced", that is, the degrees are almost the same  and the distribution of the co-neighbourhoods are almost the same as well. In some sense, this is partially true for certain cases while it is very far from being true in other cases.

\begin{theorem}\label{csillag1} Suppose that $a\geq s$. Then  
\[\sex({n, K_{1,s}: m, K_{a,b} })\geq \binom{n}{a}\binom{  \frac{n}{a!}\left(\sqrt[s]{\frac{s!m}{n}}-a+1\right)^a{\binom{n}{a}^{-1}}}{b}.\]
\end{theorem}

\begin{proof}

Let $G$ be a graph with at least $m$ copies of $K_{1,s}$. This means that $$m\ge \sum_{v\in V(G)} \binom{d(v)}{s}.$$

We may fix first a set $A$ of cardinality $a$ in $V(G)$ and count  those sets in $V(G)$ which have exactly $b$ vertices, which are all adjacent to the vertices of $A$; then we can sum this up to all $a$-sets. This provides 
	
	$$\cN(K_{a,b},G)=\sum_{A\subseteq V(G), |A|=a} \binom{d(A)}{b}.$$
	
	By applying Jensen's inequality, we get
	
		$$\cN(K_{a,b},G)\geq  \binom{n}{a}\binom{   {\sum_{A\subseteq V(G),|A|=a}{d(A)}}{\binom{n}{a}^{-1}}}{b}.$$
	
	Note that the sum of the co-degrees counts the number of ways one can choose a vertex from $V(G)$ and $a$ neighbours of that vertex. Hence we obtain
	$$\sum_{A\subseteq V(G),|A|=a}{d(A)}= \sum_{y\in V(G)}\binom{d(y)}{a}.$$
Observe that if $a\geq s$, then we can compare $\sum_{y\in V(G)}\binom{d(y)}{a}$ and $\sum_{v\in V(G)}\binom{d(v)}{s}$ by the application of Lemma \ref{technical}, which gives

$$\sum_{y\in V(G)}\binom{d(y)}{a}\geq \frac{n}{a!}\left(\sqrt[s]{\frac{s!m}{n}}-a+1\right)^a,$$ provided that the expression in the bracket is at least zero, 
and this in turn implies 

	$$\cN(K_{a,b},G)\geq  \binom{n}{a}\binom{   {\sum_{A\subseteq V(G),|A|=a}{d(A)}}{\binom{n}{a}^{-1}}}{b}\geq 
	\binom{n}{a}\binom{  \frac{n}{a!}\left(\sqrt[s]{\frac{s!m}{n}}-a+1\right)^a{\binom{n}{a}^{-1}}}{b}.$$

\end{proof}

\begin{remark}
Theorem \ref{csillag1} can be asymptotically tight if almost all the co-neighbourhoods are of the same size for the sets of size $s$ or $a$, that is, if the graph is 'balanced' in the above mentioned way. In fact, the bipartite polarity graph shows that it is indeed possible for some values of the parameters, meaning that the theorem is asymptotically sharp.
\end{remark}

Suppose now that $b\leq a< s$.
Unlike in the case of Theorem \ref{csillag1}, we conjecture that  the tight bound is obtained in two really differently structured graph family in terms of the graph cardinality $m$ with a phase transition.

\begin{conj}\label{star} Suppose now that $b\leq a< s$, and $m\gg n^{1+s-1}$. Then 
   $$ \sex(n, K_{1,s}: ~m,~ K_{a,b})=  (1+o(1))\min \{ \cN(K_{a,b}, K_q^*), \cN(K_{a,b}, \overline {K_r^*})\},$$ where $q=\min\{t\in \mathbb{Z}: \cN(K_{1,s}, K_t)\geq m\} $ and $r= \min\{t\in \mathbb{Z}: \cN(K_{1,s}, \overline {K_t})>m\}$.\end{conj}

Note that there exists a constant $\zeta=\zeta(a,b,s)>0$ such that $\forall \varepsilon>0$, there exist a threshold $n_0$ for which  if $n_0<n$ and $m<(\zeta-~\varepsilon)n^{a+b}$  then $\cN(K_{a,b}, K_q^*)> \cN(K_{a,b}, \overline {K_r^*})$ while if $n_0<n$ and $m>(\zeta-~\varepsilon)n^{a+b}$  then $\cN(K_{a,b}, K_q^*)< \cN(K_{a,b}, \overline {K_r^*})$.

This general statement is known to be true in case $a=b=1$, which is equivalent to maximizing the number of copies of $K_{1,s}$ if the number of edges is given. The case $s=2$ coincides with the well known result of Ahlswede and Katona \cite{AhlKatona} (note that they proved an exact result). This has been generalized by Reiher and Wagner \cite{ReiherWagner} for every star.



\begin{theorem}[Reiher, Wagner \cite{ReiherWagner}]
Given nonnegative integers $n$ and $m$ and an integer $k \geq 2$, the maximum number of copies of the star $K_{1,k}$ in a graph with $n$ vertices and $m$ edges is
$$\max \big( \gamma^{(k+1)/2}, \eta+(1-\eta)\eta^k \big) \frac{n^{k+1}}{k!} + O(n^k),$$
where $\gamma = m/\binom{n}{2}$ is the edge density and $\eta=1-\sqrt{1-\gamma}$, thus the maximum is attained asymptotically either on the quasi-clique or on the quasi-star.
\end{theorem}

Our Conjecture \ref{star} is verified  in certain domain in the paper of Gerbner, Patkós, Nagy and Vizer \cite{GPNV} for $a=b=1$. Their result in our terminology reads as follows.

\begin{theorem} For every graph $H$, there exists a threshold $m_0=cn^{|V(H)|}$, $c<1$ such that for $m>m_0$, $\sex(n,~ H :~ m,~ K_2)$ is asymptotically attained on the corresponding quasi-star.
\end{theorem}

Note that while one  may suggest that these graphs could provide exclusively the extremal graphs asymptotically, but it is not the case in general according to the results of \cite{Day} and \cite{Blekherman}.

Finally we mention  some previously known partial results concerning complete bipartite graphs. 

Conlon, Fox and Sudakov \cite{Conlon} proved the result which is related to a generalized version of Sidorenko's conjecture via the dependent random choice technique.

\begin{theorem}[\cite{Conlon}]
 Let $H$ be a bipartite graph with $m$ edges which has $r \geq 1$ vertices
in the first part complete to the second part, and the minimum degree in the first
part is at least $d$. Then the homomorphishm density $t_H(G):= \frac{\hom_H(G)}{|V(G)|^{|V(H)|}}$ exceeds the respective power of the  homomorphism density $(t_{K_{r,d}}(G))^{\frac{m}{rd}}$ of the complete bipartite graph $K_{r,d}$.
\end{theorem}
This gives back the Sidorenko-conjecture by taking $r = 1$ and $d = 1$. As we mentioned before, the number $\hom_H(G)$ of homomorphisms of $H$ is asymtotically equal to the number of copies of $H$ if  $\hom_H(G)$  is large enough, thus this provides a supersaturation-extremal bound.

\section{Paths versus paths; trees versus paths}

A paper of Erdős and Simonovits \cite{ES82} discusses the supersaturation of paths $P_k$ of fixed length when the number of edges is reasonably large, thus the number of paths $P_k$ and the number of walks of length $k$ are of the same order of magnitude. Their result
has been extended for trees by Mubayi and Verstaete \cite{Mubayi}.
In the unified notation, the result reads as follows.


\begin{theorem}[Mubayi-Versta\"ete, \cite{Mubayi}] Let $T_{t}$ be any tree on $t$ vertices and suppose that $d\geq t$. Then
 $$\sex{(n, P_2 : dn/2, T_t)}\geq  (1-\delta)\frac{1}{|Aut(T_t)|}nd(d-1)\cdots(d-t+2)$$ where $\delta:=\delta(d)\rightarrow 0$ as $d\rightarrow \infty$ and $Aut(T_t)$ denotes the group of automorphisms of the tree $T_t$.
\end{theorem} 

This provides a sharp result  up to the factor $1 - \delta$ in view of the graph comprising
$n/(d + 1)$ cliques of order $d + 1$.

Győri, Salia, Tompkins and Zamora \cite{GySTZ} determined asymptotically the threshold from which the function $\sex(n, P_k:m, P_t)$ exceeds $0$ for all $k, t$, i.e., they solved the Generalized Turán problem for paths. 
The order of magnitude of this threshold is in general determined by Letzter \cite{Letzter} for every $T_t$. 

Here we discuss the common generalisation of the above results  where we are given the number of paths of length  $k\geq 2$. In order to do that, we introduce the weighted version of the subgraph counter function.

\begin{nota}\label{weights}
Suppose we have a weight function $w:E(G)\rightarrow \mathbb{R}^+$ on the edges of $G$. Then the weight $w(F)$ of a subgraph $F$ is defined as the product of the weight of its edges, while $F_{(w)}(n,G)$  denotes the sum of the weights of all subgraphs of $G$ which are isomorphic to $F$, i.e. $$F_{(w)}(n,G)= \sum_{H\subseteq G, H\sim F} \prod_{e\in E(H)} w(e).$$
\end{nota}

As it was noticed later, the supersaturation result of Erdős and Simonovits mentioned at the beginning of this subsection follows from an earlier theorem about matrix inequalities, due to Blakley and Roy.

\begin{theorem}[Blakley-Roy, \cite{blakley}]\label{blak}
 For
any positive integer $q>1$, non-negative real $n$-vector $u$ and
non-negative real symmetric $n\times n$-matrix $S$, $$\langle u, Su \rangle^{q} \leq \langle u, u \rangle^{q-1}\langle u, S^{q}u \rangle $$ holds. 
\end{theorem}

Observe that for $S$ being the adjacency matrix of a graph and $u$ being the all-one vector, we get back the variant of Erdős and Simonovits on the number of walks of length $q$ compared to the average degree.
If the graph is dense enough, the number of paths of length $q$  can be estimated easily using the 
 number of walks, which yields their result.

We apply the Blakley-Roy inequality for weighted paths to get the next result. This will be a leading example for the {\em concatenation method.}
Suppose that there exists a graph $H$ on a vertex set of size $|V(H)|$ with a subset of labelled vertices.  
A graph $H'$ is obtained by concatenation from a set of  copies of $H$ if $H'$ can be built up by taking the copies of $H$ and consecutively identifying some of the labelled vertices of the next copy to the corresponding labelled vertices of a previous copy.
In the theorem below the path $P_k$ plays the role of $H$ and identification between certain end-vertices provides the copy of a longer path $G=P_t$ with $t=q(k-1)+1$. 
Note that the number of copies of $H'=P_t$ in $G$ can be expressed via Notation \ref{weights} in the following way. Let us assign a graph to $G$ on the same vertex set $V(G)$ and let the weight of an edge  $uv$ be the number of paths isomorphic to $P_k$ in $G$ for which the endvertices are $u$ and $v$.
Then the weight of the subgraph $P_q$ in the weighted graph is equal to the number of paths $P_t$ in $G$. \\
Later on, similar concatenation yields copies of cycles or Theta graph from paths if the labelled vertices of the paths which are identified are the endvertices of the paths.

\begin{theorem}\label{pathpath} Suppose that $m\gg n^{k-\frac{1}{q}}$ and $t=q(k-1)+1$ for some integer $q>1$. Then $$\sex{(n, P_k : m, P_t)}=\left(\frac{1}{2}+o(1)\right)\frac{(2m)^q}{n^{q-1}}.$$
\end{theorem}

\begin{proof} Take the $n\times n$ matrix $S$ where each entry $S_{i,j}$ equals the number of paths of length $k-1$ starting from vertex $v_i$ and ending at $v_j$ in an $n$ vertex graph $G$ which contains $m'\ge m$ paths $P_k$. Let $u\in \mathbb{R}^n$ be the all-one vector. Then the Blakley-Roy inequality reads as follows.

$$\langle {\bf 1}, S\cdot{\bf1} \rangle^{q} \leq \langle {\bf 1}, {\bf 1} \rangle^{q-1}\langle {\bf 1}, S^{q}{\bf 1} \rangle. $$

The left hand side clearly equals $(2m')^q$ while $\langle {\bf 1}, {\bf 1} \rangle=n$. Notice that  $\langle {\bf 1}, S^{q}{\bf 1} \rangle$ counts certain walks of length $q(k-1)$, in fact those which are gained by the concatenation of $q$ paths of length $k-1$ where the starting vertex of a path coincides with the end-vertex of the preceding path.
By the assumed lower bound on $m$, we get that $n^{q(k-1)} \ll \langle {\bf 1}, S^{q}{\bf 1} \rangle.$ However, the number of walks of length $t$ which contain some vertices multiple times is $O(n^{q(k-1)})$, thus twice the number of paths of length $t$ is at least $(1-o(1))\langle {\bf 1}, S^{q}{\bf 1} \rangle$. This in turn implies the lower bound part of the statement. Notice that a regular graph having asymptotically $m$ paths of length $k-1$ implies the upper bound part as well. 
\end{proof}

\begin{remark} The lower bound on $m$ in the assumption is far from being optimal, we conjecture that the statement also holds if $m\gg n$ (and $q$ is a fixed number). However to keep the focus of this paper we decided not to study the optimal bound here. 
\end{remark}

The above described concatenation also seems useful if an arbitrary tree $T_t$ is gained from a $k-2$-subdivision of a tree $T_{q+1}$, i.e. if every edge of  $T_{q+1}$ is subdivided by $k-2$ inner vertices, and one aims to determine $\sex(n, P_k : m, T_t)$, to get a similar result. The subdivision implies that $T_t$ can be built from paths of length $k-1$ such that only starting and end-vertices of the paths are identified. However the probabilistic method of \cite{dellamonica} and \cite{Mubayi} is slightly more involved to be discussed here in details.

We end this section with a conjecture.

\begin{conj} Suppose that $k>\ell$. Then
$\sex(n, P_k : m, P_{\ell})$ is attained asymptotically on the quasi-star or the quasi-clique.
\end{conj}

We remark that some cases of this conjecture are covered by Theorem \ref{spanning}.
This general case seems rather involved, and the parity of the parameters $(k,\ell)$ plays a crucial role. For illustration, see Table \ref{quasi-t}, which compares the number of paths of given length in  quasi-stars,  quasi-cliques and complete bipartite graphs.

\begin{center}

\begin{table}[h!] \label{quasi-t}

\addtolength{\leftskip}{-1cm}

\addtolength{\rightskip}{-1cm}

\begin{tabular}{|l|c|c|c|}

\hline

\quad \quad Graph families & \shortstack{Clique \\ $K_{\lambda n}$} & \shortstack{Co-clique (Quasi-star)\\  $\overline{K_{(1-\lambda)n}^*}$} & \shortstack{Bipartite graph\\ $K_{\lambda n, (1-\lambda)n}$} \\

\hline

 & & & \\
\shortstack{\# of $P_k$  in terms of $\lambda$,\\ $\frac{1}{n}\ll \lambda \leq 1$, \\ main term} & $\frac{1}{2}(\lambda n)^{k}$ & $\frac{1}{2}\sum_{t=0}^{\lceil k/2\rceil}\lambda^{k-t}(1-\lambda)^t\binom{k-t+1}{t}n^{k}$ & \shortstack{$(\lambda(1-\lambda))^{k/2}n^{k}$ for $k$ even \\   $\frac{1}{2}(\lambda(1-\lambda))^{(k-1)/2}n^{k}$ for $k$ odd }\\ 
 & & & \\

\hline

 & & & \\
\shortstack{ main term of \# of $P_k$\\ supposing that $\lambda=o(1)$ }  &  $\frac{1}{2}(\lambda n)^{k}$ & \shortstack{$\frac{1}{2}(k/2+1)\lambda^{k/2}n^{k}$ for $k$ even \\   $\frac{1}{2}\lambda^{(k-1)/2}n^{k}$ for $k$ odd } & \shortstack{$\lambda^{k/2}n^{k}$ for $k$ even \\   $\frac{1}{2}\lambda^{(k-1)/2}n^{k}$ for $k$ odd }\\ 
 & & & \\

\hline

 & & & \\
\shortstack{main term of \# of $P_k$\\ supposing that $1-\lambda=o(1)$ }  &  $\frac{1}{2}(\lambda n)^{k}$ & $\frac{1}{2}(\lambda n)^{k}$ &  \\ 
 & & & \\

\hline

\end{tabular}
\caption{ Number of paths in quasi-stars, quasi-cliques and complete bipartite graph}\label{quasi-t}
\end{table}

\end{center}

Inequalities between the number of walks of given length in general graphs has been discussed by several authors, see e.g. \cite{Hem}. A notable spectral result in this direction is due to Nikiforov \cite{Nikiforov}, who showed that $\frac{w_{k+r}(G)}{w_k(G)}\leq \lambda_1(G)^r$, for all even integers $k$ and all $r\geq 1$. Here $w_k$ denotes the number of walk of length $k$ and $\lambda_1$ is the largest eigenvalue of $G$. The case when $k$ is odd was posed as an open problem by Nikiforov \cite{Nikiforov}.

The correspondence (taking aside lower order terms) between the number $m$ of walks and  $m'$ of paths if $m$ is large enough points out the relevance between these open problems.


\section{Paths versus cycles or $\Theta$-graphs}

In this section we investigate further examples of the generalization of the original supersaturation problem of graphs, namely when we have a lower bound on the number of paths $P_{k+1}$ ($k>2$) and seek the minimum number of subgraphs $F$ where $F$ is a cycle of even length or a Theta-graph. This concept generalizes earlier results of Simonovits and partially of Jiang and Yepremyan, see \cite{JiangYep}. Gishboliner and Shapira \cite{Gish-Sh} determined the order of magnitude of $\ex(n,P_k,C_t)$. Asymptotic values were determined for some cases in \cite{GGyMV17,gp2}. Our main contribution is the following.

\begin{theorem}\label{pathcycle} Suppose that $m\gg n^{k}$. Then 
$$\sex(n,P_{k+1}:m,C_{2k})\geq \left(\frac{1}{k}-o(1)\right)\left(\frac{m}{n}\right)^2.$$
\end{theorem}

Note that if $k$ is even, then the number of paths $\cN(P_{k+1}, K_{k-1,n-k+1})=\Theta(n^{k/2+1})$  while the graph $K_{k-1,n-k+1}$ does not contain any cycle of length $2k$, thus a reasonably large lower bound on $m$ is necessary. Also, $\cN(P_{k+1}, K_{a,n-a})=\Theta(n^{k/2+1}a^{k/2})$ while  $\cN(C_{2k}, K_{a,n-a})=\Theta(n^{k}a^{k})$, which means that the order of magnitude of the above bound is best possible and in fact it is asymptotically optimal also if $n^{\frac{1}{2}-\frac{1}{k}}\ll a\ll n$.

If, on the other hand, $k$ is odd then the order of magnitude is smaller as $\cN(P_{k+1}, K_{k-1,n-k+1})=\Theta(n^{(k+1)/2})$ holds and the complete bipartite graph will not provide optimal structure. 
\begin{proof}[Proof of Theorem \ref{pathcycle}]
Choose an arbitrary pair $u, v$ in an $n$-vertex graph $G$ containing at least $m$ paths $P_{k+1}$. 
Consider the paths $P_{k+1}$ with end-vertices $\{u,v\}$ and denote their number by $m(u,v)$. We assign a $(k-1)$-uniform set system $\mathcal{F}_{\{u,v\}}$ to the set of these paths by taking the set of inner vertices of each. This provides $m(u,v)$ sets in $\mathcal{F}_{\{u,v\}}$, of which at least $\frac{m(u,v)}{(k-1)!}$ sets are distinct. 
Observe that disjoint pairs of sets in $\mathcal{F}_{\{u,v\}}$ correspond to distinct cycles $C_{2k}$ in $G$.  Thus we need a lower bound on the number of disjoint pairs in $\mathcal{F}_{\{u,v\}}$, given the cardinality of $\mathcal{F}_{\{u,v\}}$. 
This problem of supersaturation, related to the famous Erdős-Ko-Rado problem has a long history, however sharp bounds are still not known in general, see \cite{BDLST}. 
We apply the recent bound of Frankl, Kohakayawa and Rödl \cite{Frankl} which provides an asymptotically sharp bound.

\begin{lemma}\label{Fr}
 Let $\beta>1$ be a real number, $k\ll \sqrt{n}$ and suppose that $|\mathcal{F}_{\{u,v\}}|=\beta(k-1)!\binom{n-1}{k-2}$. Then the number of disjoint pairs in $\mathcal{F}_{\{u,v\}}$ is at least $$\ddp(\mathcal{F}_{\{u,v\}})\geq \left(\binom{\beta}{2} + (\beta-\lfloor \beta \rfloor)\lfloor \beta \rfloor-o(1)\right)\binom{n-1}{k-2}^2(k-1)!^2.$$
\end{lemma}


This lemma implies that the number of cycles $C_{2k}$  having vertices $u$ and $v$ in distance $k$ is at least $$\left(1-o(1)\right)\binom{m(u,v)}{2},$$
provided that $m(u,v)\gg n^{k-2}.$

Observe that 
there may exist $(k-1)!\binom{n-1}{k-2}$ intersecting $(k-1)$-sets  according to the Erdős-Ko-Rado theorem, thus we cannot guarantee any edge below the magnitude of $n^{k-2}$ via this method.

Finally, suppose that $m\gg n^k$. We drop those pairs of points for which $m(u,v)\gg n^{k-2}$ does not hold and use the Jensen's inequality to bound  $$\sum _{\substack{\{u,v\} \in \binom{V(G)}{2} \\ m(u,v)\gg n^{k-2} }}\left(1-o(1)\right)\binom{m(u,v)}{2}\geq (1-o(1)){\binom{n}{2}} \binom{m/\binom{n}{2}}{2}.$$
The derived lower bound may count every $C_{2k}$ $k$ times, thus the number of $2k$-cycles is at least $$ \left(\frac{1}{k}-o(1)\right){\binom{n}{2}} \binom{m/\binom{n}{2}}{2},$$ which completes the proof.
\end{proof}

\begin{remark}
It is easy to see that applying an extremal hypergraph result on the minimum number of  disjoint $t$-systems of $(r-1)$-sets, one can derive a similar statement on the number of Theta-graphs, obtained from identifying the end-vertices of $t$ paths $P_{k+1}$ of length $k$.
\end{remark}

In view of our previous results on $\sex(n,P_k : m, P_{\ell})$, we may count cycles $C_{t}$ in terms of the number of paths $P_{k+1}$ for further pairs $(k,t)$.

\begin{cor} Suppose that $m\gg n^{k+1-\frac{1}{q}}$ Then
 $$\sex(n,P_{k+1}:m,C_{2qk})\geq \left(\frac{1}{4k}+o(1)\right)\left(\frac{2m}{n}\right)^{2q}.$$
\end{cor}

\begin{proof}
Apply Theorem \ref{pathcycle} after Theorem \ref{pathpath} and it gives the desired result.
\end{proof}

\section{Paths versus $K_{2,t}$}
\begin{theorem}\label{pk2t}
Suppose that $t \ge k$. Then we have $$\sex(n,  P_{2k+1}: m,K_{2,t})= \Omega\left(n^2 \left(\frac{m}{n^{k+1} } \right)^{t/k}\right),$$ provided that $m\gg n^{k+1}$.
\end{theorem}

First we mention that for fixed parameters $k, t$, in polarity graphs or in suitably chosen Füredi-graphs $G$  we have $\cN(K_{2,t},G)=0$ while 
$\cN(P_{2k+1},G)=\Theta(n^{k+1})$ in terms of the order $n$ of the graph $G$.

\begin{proof}

Let us consider a graph $G$ on $n$ vertices.  Let us recall that for different vertices $v, u$ of $G$ we use  $d(v,u):=| N(v) \cap N(u)|$.  
$P_{2k+1}$ can be obtained by the concatenation of $k$ paths of length two.  This enables us to estimate the number of paths $P_{2k+1}$ from above as follows.

$$m \le \cN(P_{2k+1},G) \le \frac{1}{2}\sum_{v_1,\ldots, v_k, v_{k+1} \in {V(G)}} d(v_1,v_2) \cdot d(v_2,v_3)\cdot \ldots \cdot d(v_{k},v_{k+1}) \le$$ $$ \le \frac{1}{2}\sum_{v_1,\ldots, v_{k+1} \in V(G)} \frac{d(v_1,v_2)^{k} + \ldots + d(v_{k},v_{k+1})^{k}}{k}.$$
Observe that 
$$\sum_{v_1,\ldots, v_{k+1} \in V(G)} \frac{d(v_1,v_2)^{k} + \ldots + d(v_{k},v_{k+1})^{k}}{k} =  \binom{n-2}{k-1} \sum_{u,v \in V(G)} d(u,v)^{k}.$$
Using that $t \ge k$, the power mean inequality implies that
$$\left ( \frac{\sum_{u,v \in V(G)} d(u,v)^{k}}{{n \choose 2}} \right )^{\frac{1}{k}} \le \left ( \frac{\sum_{u,v \in V(G)} d(u,v)^{t}}{{n \choose 2}} \right )^{\frac{1}{t}},$$
hence $$m\leq  \frac{1}{2}\binom{n-2}{k-1} {n \choose 2}\left ( \frac{\sum_{u,v \in V(G)} d(u,v)^{t}}{{n \choose 2}} \right )^{\frac{k}{t}} $$
holds. Since $$\cN(K_{2,t},G)=\sum_{\substack{u,v \in V(G) \\ u \neq v}} \binom{d(u,v)}{t},$$ if $t>2$ and the half of the expression in the case $t=2$, thus we have $$\sum_{u,v \in V(G)} d(u,v)^{t} = t! (1+o(1)) \cN(K_{2,t},G)$$ for $t>2$, provided that $m\gg n^{k+1}$.
$$\sex(n,  P_{2k+1}: m,K_{2,t})= \Omega\left(n^2 \left(\frac{m}{n^{k+1} } \right)^{t/k}\right).$$
\end{proof}

To get an upper bound we consider the so-called Füredi-graphs, see \cite{fredi}, on which the maximum number of edges for  $K_{2,{r+1}}$-free graphs is attained asymptotically. For a prime number $p$ and an integer $r$ which is a divisor of $p-1$, the graph $H(p,r)$ has $\frac{p(p-1)}{r}$ vertices, and apart from $p-1$ vertices of degree $p-2$, all the vertices have degree $p-1$. It is not hard to see (c.f. \cite{ZNagy}) that  
$\cN(K_{2,t}, H(p,r))= \Theta(\frac{p^4}{r^2}\binom{r}{t} )$ while $\cN(P_{2k+1}, H(p,r))= \Theta(\frac{p^{2k+2}}{r})$. This shows that Theorem \ref{pk2t} is sharp in a certain range of $m$.

\section{An application to generalized Tur\'an problems}

A well-known technique in proving Tur\'an-type results uses supersaturation. In order to show $ex(n,F)\le x$ for some $x$, one takes another graph $H$, gives a lower bound on the number of copies of $H$ in graphs with $x$ edges, and if that is larger than $ex(n,H,F)$, then we are done (an example is the K\H ov\'ari-T. S\'os-Tur\'an theorem, where the other graph is a star).

Here we show an example, where the generalized supersaturation can be used to prove a bound for the generalized Tur\'an function.

\begin{proposition}\label{trivil} Let $s\le t$ and $m=\binom{n}{t}x$. Then
 $\sex(n, K_{q,t}: m, K_{q,s})\ge\binom{n}{s}x$. 
\end{proposition}
\begin{remark}
This bound is asymptotically sharp in the case $q\leq s$, as shown by the unbalanced complete bipartite graph.\end{remark}

\begin{proof} Let $G$ be an $n$-vertex graph with at least $m$ copies of $K_{q,t}$.
For a $t$-set $T$ of vertices in $G$, let $d(T)$ denote the number of vertices connected to each vertex of $T$. Then we have $m\le \sum_{T\subset V(G),|T|=t}\binom{d(T)}{q}$. For each such $t$-set, we obtain 
$\binom{t}{s}\binom{d(T)}{q}$ copies of $K_{q,s}$ such that the part of size $s$ is in $T$. We count each copy of $K_{q,s}$ at most $\binom{n-s-q}{t-s}<\binom{n-s}{t-s}$ ways. Therefore, the number of copies of $K_{q,s}$ is at least $m\binom{t}{s}/\binom{n-s}{t-s}=\binom{n}{s}x$.
\end{proof}

\begin{prop} Suppose that $s\leq t, r\leq q$. Then
 $\sex(n, K_{q,t}: m, K_{r,s})\ge \frac{\binom{n}{r}\binom{n}{s}}{\binom{n}{t}\binom{n}{q}}m$. 
\end{prop}

\begin{proof} We apply Proposition \ref{trivil} twice for the pairs $s,t$ and $r,q$.
\end{proof}

Gerbner, Gy\H ori, Methuku and Vizer \cite{GGyMV17} proved $\ex(n,C_4,C_{2k})=(1+o(1))\binom{k-1}{2}\binom{n}{2}$. Using Proposition \ref{trivil}, we can prove the following.

\begin{cor}
$\ex(n,K_{2,t},C_{2k})=(1+o(1))\binom{k-1}{2}\binom{n}{t}$.
\end{cor}

\begin{proof} The lower bound is given by $K_{k-1,n-k+1}$.
To show the upper bound, let $G$ be a $C_{2k}$-free graph on $n$ vertices. By the result mentioned above, is contains at most $(1+o(1))\binom{k-1}{2}\binom{n}{2}$ copies of $C_4$. Then Proposition \ref{trivil} with $q=s=2$ finishes the proof.
\end{proof}

For the sake of completeness, let us discuss $\ex(n,K_{2,t},C_{2k+1})$. Gy\H ori, Pach and Simonovits \cite{gypl} showed that for any complete multipartite graph $H$ and any $\ell$, $\ex(n,H,K_\ell)=\cN(H,G)$ for some complete multipartite graph $G$. In particular, $G=K_{a,n-a}$ for some $a$ depending on $t$. Using this, a straightforward optimisation gives an exact result for $\ex(n,K_{2,t},C_{3})$. Gerbner and Palmer \cite{gp2} showed $\ex(n,H,F)\le \ex(n,H,K_k)+o(n^{|V(H)|})$ for any graph $F$ with chromatic number $k$. In particular, this implies the asymptotic result $\ex(n,K_{2,t},C_{2k+1})=(1+o(1))\ex(n,K_{2,t},C_{3})$.

\section{Connection to Berge hypergraphs}

An area closely related to generalized Tur\'an problems is the theory of Berge hypergraphs. Extending the definition of hypergraph cycles due to Berge, Gerbner and Palmer \cite{gp1} defined Berge copies of every graph.

Given a graph $G$ and a hypergraph $\cG$, we say that $\cG$ is a \textit{Berge copy of $G$} (in short: Berge-$G$) if $V(G)\subset V(\cG)$ and there is a bijection $f:E(G)\rightarrow E(\cG)$ such that for every $e\in E(G)$ we have $e\subset f(e)$. We say that $G$ is the \textit{core} of the Berge-$G$. Observe that a graph $G$ may be the core of multiple non-isomorphic hypergraphs at the same time, and a hypergraph $\cG$ may have multiple cores, that may be isomorphic or non-isomorphic. 

Let $\ex_r(n,\text{Berge-}F)$ denote the largest number of hyperedges in an $r$-uniform hypergraph that does not contain any Berge copy of $F$. 
Chapter 5.2.2 of \cite{gp} contains a short survey of the topic.

The connection of generalized Tur\'an problems and Berge hypergraphs was established by Gerbner and Palmer \cite{gp2} with the following statement: for any graph $F$ we have $\ex(n,K_r,F)\le \ex_r(n,\text{Berge-}F)\le \ex(n,K_r,F)+\ex(n,F)$. See \cite{fkl, gmp} for more details on this connection. In this section we investigate how this connection extends to supersaturation problems.

Observe that there are multiple ways to count copies of a Berge-$F$ in a hypergraph $\mathcal{H}$. The most natural is to count subhypergraphs of $\mathcal{H}$ that happen to be Berge-$F$. However, each Berge-$F$ has a core copy of $F$ in the shadow graph (i.e. the graph consisting of the 2-element subsets of the hyperedges of $\cH$). One can count such subgraphs of the shadow graph; those that can be extended to a Berge-$F$. Note that such a core can be extended to a Berge-$F$ multiple different ways potentially, and similarly for a Berge-$F$ there may be multiple cores in the shadow graph that it extends. This motivates a third way of counting Berge-$F$'s: we count each copy of $F$ in the shadow graph, multiplied by the number of ways it can be extended to a Berge-$F$ in $\mathcal{H}$.

In other words, we can consider the bijective functions defining Berge-$F's$. Such a function $f$ maps subgraphs isomorphic to $F$ of the shadow graph of $\cH$, to subhypergraphs of $\mathcal{H}$, with the property that $e\subset f(e)$ for every edge of the shadow graph. Then we can count the images of such function, we denote their number by $\cN_1(\cH,\textup{Berge-}F)$. We can count the preimages of such function, we denote their number by $\cN_2(\cH,\textup{Berge-}F)$. Finally, we can count all those functions,  we denote their number by $\cN_3(\cH,\textup{Berge-}F)$. 

For example, if $\cH$ consists of all the four $3$-sets on $4$ vertices, and $F$ is the cherry $P_3$, then we have $\cN_1(\cH,\textup{Berge-}F)=6$, $\cN_2(\cH,\textup{Berge-}F)=12$ and $\cN_3(\cH,\textup{Berge-}F)=36$.
Obviously for any $F$ and $\cH$ we have $\cN_1(\cH,\textup{Berge-}F)\le \cN_3(\cH,\textup{Berge-}F)$ and $\cN_2(\cH,\textup{Berge-}F)\le \cN_3(\cH,\textup{Berge-}F)$. On the other hand, $\cN_3(\cH,\textup{Berge-}F)\le \binom{r}{2}^{|E(F)|} \cN_1(\cH,\textup{Berge-}F)$, as starting from a Berge-$F$ subhypergraph, for each hyperedge, we have to pick one of its subedges to get a copy of $F$ in the shadow graph. As often we are only interested in the order of magnitude in $n$ and consider $F$ and $r$ fixed, in those cases $\cN_1$ and $\cN_3$ can be considered roughly the same. This is not the case with $\cN_2$. For example, if $\cH$ has $3n+3$ vertices $a,b,c,x_1,\dots,x_n,y_1,\dots,y_n,z_1,\dots,z_n$ and $3n$ hyperedges $\{a,b,x_i\}$, $\{a,c,y_i\}$ and $\{b,c,z_i\}$ for every $i\le n$, then there is only one triangle $abc$ in the shadow graph, thus $\cN_2(\cH,\textup{Berge-}K_3)=1$. On the other hand $\cN_1(\cH,\textup{Berge-}K_3)=\cN_3(\cH,\textup{Berge-}K_3)=n^3$.

For $1\le i\le 3$ we denote by $\sex_i(n,m,\textup{Berge$_r$-}F)$ the smallest $\cN_i(\cH,\textup{Berge-}F)$ among $r$-uniform hypergraphs $\cH$ with $m$ hyperedges on $n$ vertices.

\begin{proposition} $$ \sex(n,K_r:m, F) \geq \sex_2(n,m,\textup{Berge$_r$-}F).$$
 
\end{proposition}

\begin{proof} Let us take a graph $G$ on $n$ vertices, with at least $m$ copies of $K_r$ and $\sex(n,K_r:m, F)$ copies of $F$. Let us replace each $K_r$ with a hyperedge on the same vertices. Then the resulting hypergraph $\cH_0$ is on $n$ vertices with at least $m$ hyperedges. Let us take an arbitrary subhypergraph $\cH$ of $\cH_0$ with $m$ hyperedges. We claim that $\cN_2(\cH,\textup{Berge-}F)\le \sex(n,K_r,m, F)$. Indeed, whenever we take a copy of $F$ in the shadow graph of $\cH$, it is a copy of $F$ in $G$ (no matter if it can be extended to a Berge copy in $\cH$ or not). This finishes the proof.
\end{proof}

\begin{proposition} $$ \sex_2(n, m, \textup{Berge$_r$-}F)\geq \sex(n,K_r: m-\binom{n}{2}, F),$$ 

$$ \sex_1(n, m, \textup{Berge$_r$-}F)\geq \sex(n,K_r: m-\binom{n}{2}, F).$$

\end{proposition}

\begin{proof} Let us consider a hypergraph $\cH$ on $n$ vertices, with $m$ hyperedges such that $ \cN_2(\cH,\textup{Berge-}F)=\sex_2(n, m, \textup{Berge$_r$-}F)$. Let us choose an arbitrary order of the hyperedges, and apply the following operation. We pick a subedge from each hyperedge in order, such a way that we pick each edge of the shadow graph at most once. If it is impossible, because every subedge of the hyperedge has been picked already, we mark the hyperedge.

Let $G$ be the resulting graph. As each marked hyperedge corresponds to a $K_r$ in $G$ and each other hyperedge corresponds to an edge in $G$, we have $m=|E(\cH)|\le \cN(G,K_r)+|E(G)|\le \cN(G,K_r)+\binom{n}{2}$. On the other hand every copy of $F$ in $G$ can be extended to a Berge copy in $\cH$ because we picked the edges from separate hyperedges.  

Thus $G$ has at least $m-\binom{n}{2}$ copies of $K_r$, and at most $\sex_2(n, m, \textup{Berge$_r$-}F)$ copies of $F$, finishing the proof of the first statement.

To prove the second statement, we choose a hypergraph $\cH'$ on $n$ vertices, with $m$ edges such that $\cN_1(\cH',\textup{Berge-}F)=\sex_1(n, m,\textup{Berge$_r$-}F)$, and proceed the same way. Observe that in the resulting graph $G'$ not only every copy of $F$ can be extended to a Berge copy in $\cH'$, but these copies also consist of different sets of hyperedges, thus there is at most $\cN_1(\cH',\textup{Berge-}F)$ of them, finishing the proof.
\end{proof}

\end{document}